\documentclass{scrartcl}

\newcommand{\mytitle}{On Model Predictive Funnel Control with Equilibrium Endpoint Constraints*} 
\title{\mytitle}
\author{Jens Göbel$^{1,4}$ \and 		% 0000-0003-4816-3119
	Dario Dennstädt$^{2,3}$ \and		% 0000-0003-3238-3838
	Lukas Lanza$^{2}$ \and 			% 0000-0003-1365-7004
	Karl Worthmann$^{2}$ \and  		% 0000-0002-1450-2373
	Thomas Berger$^{3}$ \and 		% 0000-0003-3987-0831
	Tobias Damm$^{1,4}$ 			% 0000-0002-8550-8174
   \thanks{* T.~Berger, D.~Dennstädt, L.~Lanza and K.~Worthmann gratefully acknowledge funding by the Deutsche Forschungsgemeinschaft (DFG, German Research Foundation)~-- Project-ID 471539468.
   This work was developed within the European project INTEREST. The project has received funding from the European Union's Horizon Europe Research and Innovation funding programme under GA No. 101160594. Views and opinions expressed are however those of the author(s) only and do not necessarily reflect those of the European Union. Neither the European Union nor the granting authority can be held responsible for them.}
   \thanks{$^{1}$Fraunhofer Institute for Industrial Mathematics (ITWM), 67663 Kaiserslautern, Germany, 
   {\ttfamily\small jens.goebel@itwm.fraunhofer.de}}
   \thanks{$^{2}$\changed{Optimization-based Control Group, TU Ilmenau}, Ilmenau, Germany, {\ttfamily\small $\{$lukas.lanza, karl.worthmann$\}$@tu-ilmenau.de}}
   \thanks{$^{3}$Universit\"at Paderborn, Institut f\"ur Mathe\-ma\-tik, Warburger Str.~100, 33098~Paderborn, Germany,
   {\ttfamily\small  thomas.berger@math.upb.de, dario.dennstaedt@uni-paderborn.de}}
   \thanks{$^4$Rheinland-Pfälzische Technische Universität Kaiserslautern-Landau, 67663 Kaiserslautern, Germany, 
   {\ttfamily\small tdamm@rptu.de}}% 
    }
\date{\today}

\usepackage[utf8]{inputenc}
\usepackage[T1]{fontenc}
\usepackage{lmodern}
\usepackage[british]{babel}
\usepackage{microtype}
\usepackage{cite}

\usepackage{tikz}
\usetikzlibrary{arrows,arrows.meta,calc,positioning,shapes,spy,topaths,math,positioning}
\usepackage{pgfplots}
\pgfplotsset{compat=1.18} 
\usepackage{dsfont} 
\usepackage{array} 
\usepackage{longtable} 
\usepackage{csquotes} 
\usepackage{graphicx} 
\usepackage{nicefrac}
\usepackage{algpseudocodex}

\usepackage[section]{placeins} 
\usepackage[tbtags]{amsmath} 
\usepackage{amsthm} 
\usepackage{mathtools}
\usepackage{interval} \intervalconfig{soft open fences}
\usepackage{amssymb}
\usepackage{dcolumn} 
\usepackage{verbatim} 
\usepackage{subcaption}
\usepackage[normalem]{ulem} 

\usepackage[textwidth=2.7cm]{todonotes} 
\setuptodonotes{noline, fancyline, shadow, color=orange!15} 
\usepackage[colorlinks=false, linktoc=all, 
pdftitle={\mytitle},
pdfauthor={Göbel, Jens},
pdfpagelayout=SinglePage, 
]{hyperref} 
\usepackage{cleveref} 
\crefname{algoline}{line}{lines}
\newcommand{\fullref}[1]{\hyperref[{#1}]{\cref*{#1}: \enquote{\textit{\nameref*{#1}}}}} 
\newcommand{\Fullref}[1]{\hyperref[{#1}]{\Cref*{#1}: \enquote{\textit{\nameref*{#1}}}}}

\newtheorem{lemma}{Lemma}
\newtheorem{theorem}[lemma]{Theorem} 
\newtheorem{corollary}[lemma]{Corollary}

\newtheorem{algo}[lemma]{Algorithm}
\newtheorem{definition}[lemma]{Definition}

\crefname{problem}{problem}{problems}

\newcommand{\R}{\mathbb{R}}

\newcommand{\N}{\mathbb{N}}
\newcommand{\dd}[2][ ]{\tfrac{\text{\normalfont d}#1}{\text{\normalfont d}#2}}

\newcommand{\defeq}{\mathrel{\coloneqq}}

\newcommand{\T}{\mathsf{T}}

\def\changed#1{#1}

\begin{document}
    \maketitle
    
    \begin{abstract}
    We propose \emph{model predictive funnel control}, a novel model predictive control (MPC) scheme building upon recent results in funnel control. 
    The latter is a high-gain feedback methodology that achieves evolution of the measured output within predefined error margins.
    The proposed method dynamically optimizes a parameter-dependent error boundary in a receding-horizon manner, thereby combining prescribed error guarantees from funnel control with the predictive advantages of MPC.
    \changed{On the one hand, this approach promises faster optimization times due to a reduced number of decision variables, whose number does not depend on the horizon length.
    On the other hand, the continuous feedback law improves the robustness and also explicitly takes care of the inter-sampling behavior. }
    We focus on proving stability by leveraging results from MPC stability theory with terminal equality constraints. Moreover, we rigorously show initial and recursive feasibility.
    \end{abstract}

\section{Introduction}
    We address the problem of stabilizing a nonlinear 
    control system subject to time-varying output constraints given 
    by user-specific performance bounds (funnel boundaries). 
    To achieve this, we propose a novel bi-level control framework that synergies \emph{funnel control}
    and \emph{model predictive control} (MPC): 
    \begin{itemize}
        \item A lower-level fixed funnel control law ensuring constraint satisfaction.
        \item An upper-level MPC-based optimizer tuning the funnel control parameters (see~\Cref{Fig:RobustControllerScheme}).
    \end{itemize}
    \textbf{Funnel control} is a model-free controller that achieves output reference tracking, first described in \cite{Ilchmann2002}. 
    It has since garnered significant attention, 
    see, e.g., \changed{the recent survey article~\cite{berger2021funnel} and references therein. }
    \changed{Funnel control has been successfully applied in many fields of engineering. Examples include speed control of wind turbines, position control of robotic manipulators or current control of synchronous machines, cf.~\cite{Hackl2017}.}
    \changed{
    A similar concept, namely prescribed performance control (PPC), introduced in~\cite{bechlioulis2008robust} and further developed e.g.\ in~\cite{bechlioulis2014low}, has recently been linked to control barrier functions  for reactive feedback design~\cite{namerikawa2024equivalence}. 
    Both PPC and funnel control are high-gain adaptive control methods that enforce tracking within error bounds by scaling the gain with the inverse of the distance between tracking error and given boundary. 
    Under certain structural conditions, such as matched input-output dimensions and a well-defined relative degree, these methods ensure bounded tracking errors. 
    While earlier funnel control required nonzero funnel diameters, recent works~\cite{lee2019asymptotic,berger2021funnel} allow asymptotic tracking for systems with higher relative degree. 
    The present paper builds on~\cite{lanza2024exact}, which demonstrates exact tracking in finite time using funnel control.
    }
    
    \textbf{MPC} 
    is a well-established and versatile control methodology that \changed{is widely applied both in research and industry, see, e.g., \cite{Lee2011} or the recent survey~\cite{badgwell2021model}}.  
    Unlike the model-free funnel control, MPC relies on a model of the controlled process.
    It optimizes the predicted system behavior over a discretized input signal according to a pre-defined cost function.
    While powerful, the computational complexity of solving this optimization problem within strict sampling time constraints \changed{poses significant difficulties in real-time applications, see, e.g., \cite{jerez2014embedded,gros2020linear} and the references therein}. 
    Beyond that, the open-loop application of the control signal between optimization steps limits 
    its robustness against disturbances. 
    \changed{Common approaches} to ensure constraint satisfaction under external disturbances or even under model mismatch are, e.g., tube-based MPC \cite{mayne2005robust,mayne2011tube} \changed{or min-max MPC schemes, see, e.g.,~\cite{limon2006input,xie2024data}, where the latter considers a data-driven scheme}. 
    
	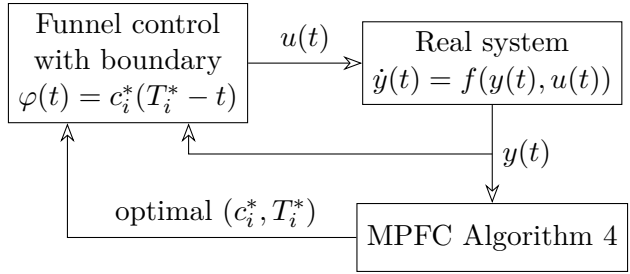
\begin{figure}[ht]
	\begin{center}
		\begin{tikzpicture} [
			>={Stealth[length=2.5mm,open]},
			align=center, 
			block/.style={draw,minimum height=0.9cm} ]
			
			\node [block] 						(System) 	{Real system \\ $\dot y(t) = f(y(t),u(t))$};
			\node [block, left=1.5cm of System] (FC) 		{Funnel control \\ \changed{with boundary} \\ $\varphi(t) = c_i^*(T_i^*-t)$};
			\node [block, below=1.3cm of System] 	(MPC) 	{\changed{MPFC~\Cref{alg:fmpc}}};
			
			\draw[->] (FC) -- node[above] {$u(t)$} (System); 
			\draw[->] (System) -- (MPC) node[pos=0.5,right] {$y(t)$} coordinate[pos=0.5] (y1); 
			\draw[->] (y1) -| (FC.-45); 
			\draw[->] (MPC) -| (FC.-135) node[above,pos=0.05,anchor=south east] {\changed{optimal~$(c_i^*,T_i^*)$}};
		\end{tikzpicture}
	\end{center}
	\caption{
            \footnotesize
            \changed{Structure of the proposed bi-level controller, i.e.,}
		feedback control for unknown real system with optimal funnel function parameters \changed{$(c,T)$}.
		The model is initialized with system \changed{output} data~$y(t_i)$. 
		Then, \changed{the MPFC~\Cref{alg:fmpc}} provides parameters~$(c_i^*,T_i^*)$ to the
		funnel controller~\eqref{eq:ExactFunnelControl}. 
        The latter is applied continuously to the system. 
	}
	\label{Fig:RobustControllerScheme}
	\end{figure}

    Our integrated method, termed \emph{model predictive funnel control} (MPFC), merges MPC and funnel control by
    optimizing a parameter-dependent funnel shape using an MPC-like strategy.
    \changed{That is, we seek an optimal feedback among a restricted family of feedback laws.
    At each MPFC step, funnel-shape parameters are optimized for the current prediction horizon in the closed-loop system. This reduces optimization variables compared to classical MPC, where the decision vector grows with the horizon length. MPFC also ensures a stabilizing control law between recalculation steps, enhancing robustness over traditional MPC, which uses open-loop control until the next iteration. Since we apply a funnel controller, performance is maintained even if the model doesn’t fully match the real system. 
    Stability is proven using equilibrium endpoint constraints from classical MPC theory, cf.~\cite{keerthi1988optimal}. This stability relies on MPC predictions reaching the system's equilibrium by the end of the prediction horizon through additional constraints, enabling a cost-neutral horizon extension and proving bounded closed-loop costs. For the MPFC version presented here, results from~\cite{lanza2024exact} naturally ensure that the equilibrium is reached within the prediction horizon, without extra constraints.
    }
    
    While prior work has explored combinations of funnel control and MPC under
    the name \emph{funnel MPC} in a series of papers, see \cite{berger2020learning,BergDenn21,Berger2024FMPC} 
    and the references therein, our approach differs fundamentally in both
    methodology and scope. 
    Existing funnel MPC implementations retain the classical funnel control
    paradigm of a static, user-defined funnel shape and instead optimize only
    the applied control values. This penalizes proximity of the system output to
    the given boundaries via rising costs but does not adapt the funnel
    itself. Other MPC approaches that consider prescribing output constraints 
    usually focus on discrete-time systems instead \cite{kohler2022constrained}.
    
    The funnel shape considered in this work, as well as in~\cite{lanza2024exact}, differs from classical funnel control in that it reaches zero in prescribed finite time~$T > 0$. 
    Due to this singularity in $T$, the closed-loop dynamics require infinitely small step sizes in the numerical solution.
    The results developed in this brief paper prove the feasibility of the proposed control algorithm from a theoretical perspective.
   
	\medskip
    {
    \textbf{Nomenclature.}
    For $x,y \in \R^n$, we use $x^\top y =: \langle x, y \rangle$, and $\|x\| = \sqrt{\langle x,x \rangle}$.
    For $V\subseteq\R^m$, we denote by $\mathcal{C}^k(V;\R^{n})$ the set of $k$-times continuously differentiable functions $f\colon V\to \R^{n}$. 
    For an interval $I\subseteq\R$ and $k\in\N$, $L^\infty(I; \R^{n})$ is the Lebesgue space of measurable, essentially bounded functions $f \colon I\to\R^n$ with norm $\|f \|_{\infty} = \text{esssup}_{t \in I} \|f(t)\|$. 
    
    }
    
    \section{System class and control objective} \label{sec:systemClassAndControlObjective}
    We introduce the system class under consideration and formulate the control objective. 
    Consider a non-linear multi-input multi-output control system
    \begin{align} \label{eq:system}
        \dot{y}(t) &= f(y(t),u(t)), \qquad y(0) = \hat{y}_0 \in \R^m,
    \end{align}
    where~$y(t) \in \R^m$ \changed{and~$u(t) \in \R^m$ denote the output and 
    the input 
    at time~$t \ge 0$, respectively.
    Note that the dimension of the input and output coincide. 
    We call a (locally) absolutely continuous function $y: [0, \omega) \to \R^m$, $\omega\in (0,\infty]$, with $y(0) =y^0 \in \R^m$ a solution (in the sense of Carathéodory) to~\eqref{eq:system}, if~$y$ satisfies~\eqref{eq:system} for almost all~$t \in [0,\omega)$. 
    A solution~$y$ is maximal, if it has no right extension that is also a solution; it is global, if $\omega=\infty$}.
    The function~$f \in \mathcal{C}(\R^m \times \R^m; \R^m)$ is assumed to satisfy the following high-gain property. 
    \begin{definition}[{High-gain property {\cite[Def. 1.2]{berger2021funnel}}}]
    \label[definition]{def:highgainproperty}
        For $m\in\N$, a function $f\in\mathcal{C}(\R^m\times\R^m;\R^m)$ satisfies the \emph{high-gain property}, if there exists $\nu\in (0,1)$ such that, for every compact $K\subset \R^m$, the continuous function 
        \begin{align*}
            \chi: \R & \to \R, \\
            s &\mapsto \min \left\{ \langle v, f(z,-sv) \rangle  |  z \in K, v \in \R^m, \nu \le \|v\| \le 1  \right\}
        \end{align*}
        satisfies $\sup_{s\in\R} \chi(s)=\infty$.
    \end{definition}
     
    \par\medskip 
    Intuitively, the high-gain property states that the system reacts fast, if the input is large.
    It is therefore a core ingredient when aiming for tracking with prescribed performance via reactive (non-predictive) feedback.
    \changed{Linear systems $\dot y(t) = A\,y(t) + B\,u(t)$ satisfy the high-gain property, if the matrix~$B$ is sign-definite, cf.~\cite[Rem.~1.3]{berger2021funnel}}.
    Using the high-gain property, we define the class of systems under consideration.
    \begin{definition}[System class $\mathcal{S}$]
        A system~\eqref{eq:system} belongs to the system class~$\mathcal{S}$, if~$f$ satisfies the high-gain property (\Cref{def:highgainproperty}) and $f(0,0)=0$.
    \end{definition}
   
    We aim to design a prediction-based feedback controller \changed{to asymptotically stabilize system~\eqref{eq:system}, i.e., to force the system to $y(t) \to 0$ for $t \to \infty$.}
    Moreover, the output should evolve within prescribed margins, i.e., 
    \begin{equation} \label{eq:ControlObjective}
       \forall\, t\ge 0:\ \|y(t)\| \le \psi(t)
    \end{equation}
    for some positive function~$\psi$ given by the control engineer. 
    This predefined performance funnel can be chosen as 
    $\psi\in\mathcal{C}^1(\R_{\geq 0}; \R)$, with $\psi(t) > 0$ for all~$t \ge 0$, 
    or as $ \psi(\cdot) \equiv \infty $ if no restrictions are posed on the transient behavior of the system.

    \section{Control Methodology} \label{Sec:ControlMethodology}
    We introduce the control methodology to achieve the control objective~\eqref{eq:ControlObjective} for systems in~$\mathcal{S}$ by combining model-free funnel control with model-based optimization. 
    To this end, we first recall the specific version of funnel control at play, before proposing model predictive funnel control. 

    \subsection{Exact Funnel Control} \label{Sec:ExactFunnelControl}
    \changed{
    To utilize well established techniques 
    from MPC stability analysis, namely equilibrium endpoint constraints, we employ a feedback law that guarantees that the system reaches the equilibrium within the prediction horizon. To this end, we use} results from~\cite{lanza2024exact}, where a feedback controller has been proposed, which achieves exact tracking in predefined finite time.
    We briefly recap the controller design~\cite{lanza2024exact}. 
	For $c,T > 0$, define the \emph{funnel boundary} by
	\begin{equation} \label{eq:FunnelBoundaryExact}
		\varphi(t;c,T) := c (T-t), \quad t \in [0,T).
	\end{equation}
    Choose a \changed{bijection~$\alpha_c\in \mathcal{C}([0,1); [2c,\infty))$, and a surjection~$N\in\mathcal{C}(\R_{\ge 0}; \R)$}. 
    Feasible choices are, e.g., $\alpha_c(s) = 2c/(1-s)$ and $N(s) = s \cos{(s)}$.
    Typically, the control direction is known ($\pm$). In this case, the choice of~$N$ simplifies to~$N(s) = \pm s$, cf.~\cite[Rem.~1.8]{berger2021funnel}.
    With these functions and parameters, we recall~\cite[Thm.~3.1]{lanza2024exact}, adapted to our setting.
    \begin{lemma}\label[lemma]{Thm:ExactFunnelControl}
        For~$\varphi$ given by~\eqref{eq:FunnelBoundaryExact}, consider a system~\eqref{eq:system} contained in~$\mathcal{S}$ with $\|\hat{y}_0\| < \varphi(0) = cT$.
        Then, the application of the feedback law
        \begin{equation} \label{eq:ExactFunnelControl}
            u(t) =  (N \circ \alpha_c)\left(\frac{\|y(t)\|^2}{\varphi(t)^2}\right) \frac{y(t)}{\varphi(t)}
        \end{equation}
        to system~\eqref{eq:system} 
        yields \changed{an absolutely continuous maximal 
        closed-loop 
        solution $y\colon [0,T)\to\R^m$ satisfying} $\| y(t) \| < \varphi(t)$ for all~$t \in [0,T)$, $\lim_{t \to T} \|y(t)\| = 0$, and 
        $u \in L^\infty([0,T];\R^m)$.
        
    \end{lemma}
    
    Note that, \cite[Thm.~3.1]{lanza2024exact} states $u \in L^\infty([0,T);\R^m)$ but the control can be extended
    to the closed interval. 

    \subsection{Model Predictive Funnel Control \changed{(MPFC)}} \label{Sec:MPFC}
    We now present model predictive funnel control, which combines the funnel controller of the previous section with model-based predictions to optimize a parameter-dependent funnel shape on a receding horizon. 

    For  symmetric positive semi-definite matrices $Q, R \in \R^{m\times m}$, define the \emph{stage costs} by 
	\begin{equation} \label{eq:stageCosts}
		\ell:\R^m\times\R^m \to \R_{\ge 0}, \ (y,u)\mapsto \langle y, Q y \rangle + \langle u , R u \rangle.
	\end{equation}
    In addition, choose a \emph{step size} $h>0$ and \emph{horizon length} $H\defeq nh$ with $n\in\N_{\geq 2}$\changed{. The sampling times are} $t_i = ih$, $i\in\N_0$. 
    \changed{We}
    define the funnel boundary by~\eqref{eq:FunnelBoundaryExact}
    and choose a surjection $N\in\mathcal{C}(\R_{\geq 0};\R)$  and a bijection
	$\alpha_c \in\mathcal{C}(\interval[open right]{0}{1}; \interval[open right]{2c}{\infty})$, depending on the parameter $c$. 
	These functions define the funnel feedback law~\eqref{eq:ExactFunnelControl} up to a specific funnel shape parameterization, which is defined by a choice from the following set: 
    For $t\geq 0$, $\hat{y}\in\R^m$, and an \textit{outer funnel function}~$\psi$, \changed{chosen by the control engineer or given by physical constraints, } define the \emph{feasible set}
	\begin{multline} \label{eq:feasibleset}
		\qquad
		\mathcal{F}_H(t,\hat{y}) \defeq 
        \Bigl\{ 
        (c,T) \in \R_{>0} \times (0,H] \bigm|
        \|\hat{y}\| < \varphi(0;c,T) ,\\
        \forall\,\tau\in[0,T)\colon \, \varphi(\tau;c,T) \le \psi(t+\tau) 
        \Bigr\} . 
        \qquad
	\end{multline}
    Consider a system \labelcref{eq:system} of class $\mathcal{S}$. For \changed{the decision variables} $(c,T)\in\mathcal{F}_H(t,\hat{y})$, we denote the \emph{cost function} by 
	\begin{subequations} \label{eq:valuefcn}
		\begin{align}
			 J_H(\hat{y},c,T) &\defeq \int_{0}^{H} \ell(y(\tau),u(\tau)) \,\mathrm{d}\tau +c \\
			 \text{s.t. } \dot{y}(\tau) \label{eq:systemfunction}
			&= f(y(\tau),u(\tau)), \quad y(0) = \hat{y}, \\
			\text{and } u(\tau) \label{eq:controllaw}
			&=  \begin{cases}
				 (N \circ \alpha_c)
				 \left( \left\| \frac{y(\tau)}{\varphi(\tau;c,T)}  \right\| ^2 \right)  \frac{y(\tau)}{\varphi(\tau;c,T)}, &  \tau<T \\
				 0,  &\text{else.}
			\end{cases} 
		\end{align}
	\end{subequations}
    \changed{For a unified notation in the case that the equilibrium is reached between two sampling steps, set $J_H(\hat{y},c,\vartheta)\defeq c$ for $ \vartheta \leq 0$ and any $c>0$. }
    Define the \emph{optimal value function} as 
    \changed{
	\[
		V_H(t, \hat{y})\defeq\inf_{(c,T)\in\mathcal{F}_H(t,\hat{y})} J_H(\hat{y},c,T).
	\]
    }
    With this, we define the control algorithm.

    \begin{algo}[Model Predictive Funnel Control]\label[algorithm]{alg:fmpc} 
    \begin{algorithmic} \ 
        \Require{
        	A system \labelcref{eq:system} of class $\mathcal{S}$,
            outer funnel boundary $\psi\in\mathcal{C}^1(\R_{\geq 0}; \R)$ with $\psi(t) > 0$ for all~$t \ge 0$, or $ \psi \equiv\infty $,
            stage cost $\ell$ as in~\eqref{eq:stageCosts}, 
            $h>0$,
            $H\defeq nh$ for $n\in\N_{\geq 2}$, 
            surjection $N\in\mathcal{C}(\R_{\geq 0};\R)$, 
            bijections $\alpha_c \in\mathcal{C}(\interval[open right]{0}{1}; \interval[open right]{2c}{\infty})$ for~$c>0$.}
        \State $i \gets 0 $
        \Loop
            \State At time $t_i$, measure the output~$y(t_i)$ of system~\eqref{eq:system} and set $\hat{y}_i = y(t_i) \in\R^m$. 
            \State \changed{Compute approximation $\hat{V}_H(t_i,\hat{y}_i)> V_H(t_i,\hat{y}_i)$.}
            \If{$i=0$} 
                \State Find $(c_0^*,T_0^*)\in\mathcal{F}_H(0,\hat{y}_0)$ such that
                \begin{align*}
                    J_H(\hat{y}_0,c_0^*,T_0^*) \leq \changed{\hat{V}_H(0,\hat{y}_0)}
                \end{align*}
            \Else 
                \State Find $(c_i^*,T_i^*)\in\mathcal{F}_H(t_i,\hat{y}_i)$ such that
                \begin{align}
                    \label{eq:fmpcstep2}
                    J_H(\hat{y}_i,c_i^*,T_i^*)\leq \min \bigl\{ J_H(\hat{y}_i,c_{i-1}^*,T_{i-1}^*-h), 
                    \changed{\hat{V}_H(t_i,\hat{y}_i)}
                    \bigr\}
                \end{align}
            \EndIf
            \State \vspace{3pt}On the interval $[t_i,t_i+h)$ apply~\labelcref{eq:controllaw}
            with $\alpha_{c_i^*}$ and $\varphi(t-t_i;c_i^*,T_i^*)$ to system~\eqref{eq:system}  
            \State $ i \gets i+1 $
        \EndLoop
    \end{algorithmic}
	\end{algo}
    \changed{
    Note that calculating the infimum~$V_H(t_i,\hat y_i)$ is not required in the algorithm (which might be intractable), but an optimization procedure will output an approximation~$\hat V_H(t_i,\hat y_i)$ of~$V_H(t_i,\hat y_i)$ in each iteration.
    }
	
\section{Main Results} \label{sec:MainResults}
    We present the two main results.
    In \Cref{lemma:FMPCAlgRecFeasible}, we prove that \Cref{alg:fmpc}
    is initially and recursively feasible for all 
    initial values $\hat{y}_0$ with $\|\hat{y}_0\|<\psi(0)$.
    Utilizing well-known results of stability via equilibrium terminal 
    constraints, cf.~\cite{keerthi1988optimal},
    \Cref{thm:TerminalConstraintsConvergence} shows that the closed-loop
    application of the algorithm is stabilizing.
    
	\begin{theorem} \label[theorem]{lemma:FMPCAlgRecFeasible}
        Let~$\psi \in \mathcal{C}^1(\R_{\ge 0};\R_{> 0})$ with $\psi(t) > 0$ for all~$t \ge 0$, or~$\psi=\infty$, be given.
        Let $y(0)$ in~\eqref{eq:system} \changed{satisfy} $\|y(0)\| < \psi(0)$. 
        Then, \Cref{alg:fmpc} is initially feasible, i.e., there exist $(c_0^*,T_0^*) \in \mathcal{F}_H(0,y(0))$ such that $J_H(y(0),c_0^*,T_0^*) \le \changed{\hat{V}_H(0,y(0))}$.
        Moreover, \Cref{alg:fmpc} is recursively feasible, meaning the solvability
        of the problem in \labelcref{eq:fmpcstep2} at time $t_i,
        i\in\N_{0}$ implies its solvability at the next time step $t_{i+1}$.
        Additionally, every maximal solution $y_i:[t_i,t_{i+1}]\to\R^m$ 
        of the initial value problem~\eqref{eq:system} with initial value $y_i(t_i)=y_{i-1}(t_i)$
        on the interval $[t_i,t_{i+1}]$ stays inside the funnel boundaries, meaning, for all~$i \in \N$ and $t \in [t_i,t_{i+1}]$, we have
		\begin{equation} \label{eq:lemma:FMPCAlgRecFeasible:outputsmallerthanfunnel} 
		    \|y_i(t)\| \le \max\left\{0, \varphi(t-t_i;c_i^*,T_i^*)\right\}
            \le  \psi(t).
		\end{equation} 
	\end{theorem}
    \medskip
 	\begin{proof}
        \emph{Step 1:} 
        Given $(\hat{t},\hat{y})\in\R_{\geq0}\times\R^m$ with $\|\hat{y}\|<\psi\changed{(\hat{t}\,)}$,
        we show that $\changed{\mathcal{F}_H(\hat{t},\hat{y})}\neq\emptyset$. 
        To see this, consider the following two cases.
        If $\psi\equiv\infty$, then $\left( \tfrac{\|\hat{y}\|+1}{H}, H \right) \in \changed{\mathcal{F}_H(\hat{t},\hat{y})}$. 
        Otherwise, invoking~$\psi \in \mathcal{C}^1(\R_{\ge 0};\R)$ with~$\psi(t) > 0$ for all~$t \ge 0$,
        we see that the candidates
        \begin{align*}
            \hat T &\defeq \min\left\{ \frac{\psi\changed{(\hat{t}\,)}+\|\hat{y}\|}{2 \| \dot\psi|_{\interval{\hat{t}}{\hat{t}+H}} \|_\infty } ,\; H \right\},  &
            \hat c &\defeq \frac{\psi\changed{(\hat{t}\,)}+\|\hat{y}\|}{2 \hat T } 
        \end{align*}
        fulfill $(\hat c,\hat T)\in\mathcal{F}_{H}(\hat{t},\hat{y})\neq\emptyset$ 
        by checking the set predicates. 
        $\hat{T}\in \interval[open left]{0}{H}$ and $\hat{c}>0$ follow directly. Furthermore, 
        \begin{align*}
            \| \hat{y} \| 
            < \varphi(0;\hat{c},\hat{T}) 
            = \hat{c}\hat{T} 
            = \tfrac{1}{2} (\psi\changed{(\hat{t}\,)}+\|\hat{y}\|) < \psi\changed{(\hat{t}\,)} ,
        \end{align*}
        where we used $\|\hat y\| < \psi\changed{(\hat{t}\,)}$.
        Since, for  $\tau \in \interval[open right]{0}{\hat{T}}$,
        $\dd{\tau} \varphi(\tau;\hat{c},\hat{T}) 
        = -\hat{c} 
        \leq -\|\dot\psi|_{\interval{\hat{t}}{\hat{t}+H}} \|_\infty$, we calculate 
        \begin{equation*}
           \varphi(\tau; \hat{c}, \hat{T}) 
           = \varphi(0; \hat{c}, \hat{T}) + \int_0^\tau \dd{s} \varphi(s;\hat{c},\hat{T}) \,ds 
           < 
           \psi\changed{(\hat{t}\,)} + \int_{\hat t}^{\hat t +\tau} \dot\psi(s) \,ds = \psi(\hat t+\tau).
        \end{equation*}
        
        \emph{Step 2:} 
        Given $(t_i,\hat{y}_i)\in\R_{\geq 0}\times \R^m$ with $\|\hat{y}_i\|<\psi(t_i)$ for $i\in\N_{0}$,
        we show $V_H(t_i,\hat{y}_i)\in\R$. 
        According to \emph{Step~1}, there exist $c,T \in \R$ such that $\varphi(t-t_i;c,T) < \psi(t)$ for all~$t \in \changed{\interval[open right]{t_i}{t_i+T}}$ \changed{and} $\|\hat{y}_i\| < \varphi(0;c,T)$.
        Then, \Cref{Thm:ExactFunnelControl} yields an input signal~$u_i \in L^\infty([t_i,t_{i}+T];\R^m)$ such that, 
        for $(c,T) \in \mathcal{F}_H(t_i,\hat{y}_i)$, the initial value problem~\eqref{eq:system} 
        with initial value $y(t_i)=\hat{y}_i$ has a solution $\tilde{y}_i:[t_i,t_i+T]\to\R^m$ with 
        $\lim_{t\to T} \|\tilde{y}_i(t_i+t)\| = 0$ and 
        \begin{equation}
            \forall\, t \in [t_i,t_i+T) \colon  
            \|\tilde{y}_i(t)\| < \varphi(t-t_i;c,T) \leq \psi(t).\label{eq:proof:FMPCAlgRecFeasible:outputsmallerthanfunnel}
        \end{equation}
        
        On $[t_i+T,t_i+H]$, the solution $\tilde y_{i}$ can be continued with $u=0$.
        Since $f(0,0)=0$ by assumption, we get $\tilde{y}_i|_{[t_i+T,t_i+H]} \equiv 0$.
        Then, $\tilde{y}_i,u_i \in L^\infty([t_i,t_i+H];\R^m)$ implies that $J_H(\hat{y}_i,c,T)$ in~\eqref{eq:valuefcn} is finite.
        Thus, $V_H(t_i,\hat{y}_i)$ is finite.
        \changed{Note that $\|\tilde{y}_i(t)\|<\psi(t)$ holds in particular at the subsequent MPC sampling time $t_{i+1} = t_i + h$}. 
        
        \emph{Step 3:} We show initial and recursive feasibility of \Cref{alg:fmpc}.
        
        \emph{Step 3.1:} 
        According to \emph{Step~2}, $V_H(0,y(0))$ is finite. 
        Thus, there exist $(c_0^*,T_0^*) \in \mathcal{F}_H(0,y(0))$ 
        such that $J_H(y(0),c_0^*,T_0^*) \le \changed{\hat{V}_H(0,y(0))}.$
        This is, \Cref{alg:fmpc} is initially feasible.
        
        \emph{Step 3.2:} 
        We show recursive feasibility.
        For~$i \in \N_0$ let $y_{i}:[t_{i},t_{i+1}] \to \R^m$ be the solution of the 
        initial value problem \eqref{eq:system} with initial value $y_i(t_i)=y_{i-1}(t_i)$
        and control~$u_i$.
        Further, let $(c_i^*,T_i^*)\in\mathcal{F}_H(t_i,y_{i}(t_i))$ be the associated funnel parameters.
        Define $\hat{y}_{i+1}:=y_i(t_{i+1})$.
        We have to show that there exist $(c_{i+1}^*,T_{i+1}^*)\in\mathcal{F}_H(t_{i+1},\hat{y}_{i+1})$ satisfying \eqref{eq:fmpcstep2}.
        \changed{
        If $T_i^* \leq h$, then, by \emph{Step 2}, the system has already reached the equilibrium at time $t_{i+1}$. 
        In this case, $V_H(t_{i+1},\hat{y}_{i+1}) = 0$ and $c^*_{i+1}$ can be chosen small enough such that $(c^*_{i+1},H) \in \mathcal{F}_H(t_{i+1},\hat{y}_{i+1})$ (similar to \emph{Step~1}) and satisfies \labelcref{eq:fmpcstep2}. 
        Otherwise, if $T_i^* > h$, then the candidates $(c_{i+1}^*,T_{i+1}^*) \defeq (c_i^*, T_i^*-h)$ fulfill $(c_{i+1}^*,T_{i+1}^*)\in\mathcal{F}_H(t_{i+1},\hat{y}_{i+1})$, which can be seen by checking the set predicates of $\mathcal{F}_H(t_{i+1},\hat{y}_{i+1})$ in combination with \emph{Step 2}.
        Also, they satisfy $J_H(\hat{y}_{i+1},c_{i+1}^*,T_{i+1}^*) \leq J_H(\hat{y}_{i+1},c_{i}^*,T_{i}^*-h)$ by definition. 
        This, together with the non-emptiness of $\mathcal{F}_H(t_{i+1},\hat{y}_{i+1})$, shows existence of a parameter pair satisfying \labelcref{eq:fmpcstep2}.
        }
        
         \emph{Step 4:}
         As a direct consequence of \eqref{eq:proof:FMPCAlgRecFeasible:outputsmallerthanfunnel} in \emph{Step~2}, 
         every solution $y_i$ satisfies \eqref{eq:lemma:FMPCAlgRecFeasible:outputsmallerthanfunnel}.
         This completes the proof.
	\end{proof}

    Next, we prove boundedness of the closed-loop costs. 
	\begin{lemma} \label[lemma]{thm:boundedCosts}
        Let the assumptions of \Cref{lemma:FMPCAlgRecFeasible} be fulfilled. 
		Then, the input signals $u_i(\cdot)$ and the output signals $y_i(\cdot)$ resulting from \Cref{alg:fmpc} fulfill
		\begin{align*}
			J_\infty^{cl}(H,y_0(0)) &\defeq 
			     \sum_{i=0}^{\infty} \int_{t_i}^{t_{i+1}}  \ell(y_i(t),u_i(t)) dt \\
			& \leq \changed{\hat{V}_H(0,y_0(0))} < \infty.
		\end{align*}
	\end{lemma}
	\begin{proof}
        Continuing the considerations of the proof of \Cref{lemma:FMPCAlgRecFeasible}, it follows that 
        for all $\hat{t}\geq 0$ and $\hat{y}\in\R^m$ we have 
        $\mathcal{F}_{H-h}(\hat{t},\hat{y})\subseteq\mathcal{F}_{H}(\hat{t},\hat{y})$ \changed{by inspecting the properties of the set in~\eqref{eq:feasibleset}}.
        Moreover, the horizon can be extended with zero cost:
        \begin{equation*}
           \forall\, (c,T)\in \mathcal{F}_{H-h}(\hat{t},\hat{y}):\ J_{H-h}(\hat{y},c,T) = J_H(\hat{y},c,T).
        \end{equation*}
        We calculate, for $i\in\N$: 
        \begin{align} 
                & J_H(y_i(t_i),c_i^*,T_i^*) 
                 =\int_{t_i}^{t_i+H} \ell(y_i(t),u_i(t)) \;dt +c^*_i \label{eq:boundedCosts:descent} \\
                & =\int_{t_i}^{t_{i+1}} \ell(y_i(t),u_i(t)) \;dt + \int_{t_{i+1}}^{t_{i}+H} \ell(y_i(t),u_i(t)) \;dt +c^*_i \notag\\
                & \stackrel{\labelcref{eq:valuefcn}}{=}\int_{t_i}^{t_{i+1}} \ell(y_i(t),u_i(t)) \;dt + J_{H-h}(y_{i+1}(t_{i+1}),c_i^*,T_i^*-h) \notag\\
                & =\int_{t_i}^{t_{i+1}} \ell(y_i(t),u_i(t)) \;dt + J_{H}(y_{i+1}(t_{i+1}),c_i^*,T_i^*-h) \notag\\
                & \stackrel{\labelcref{eq:fmpcstep2}}{\geq}\int_{t_i}^{t_{i+1}} \ell(y_i(t),u_i(t)) \;dt + J_{H}(y_{i+1}(t_{i+1}),c_{i+1}^*,T_{i+1}^*). \notag  
        \end{align}
        Rearranging gives 
        \begin{equation*}
            \int_{t_i}^{t_{i+1}} \ell(y_i(t),u_i(t)) dt 
            \leq J_H(y_i(t_i),c_i^*,T_i^*)  - J_{H}(y_{i+1}(t_{i+1}),c_{i+1}^*,T_{i+1}^*).
        \end{equation*}
        As this holds for all $i\in\N_0$, summing up to $K\in\N$ gives 
        \begin{subequations}
        \begin{align*}
            \sum_{i=0}^{K}&\int_{t_i}^{t_{i+1}} \ell(y_i(t),u_i(t)) dt \\
            &\leq J_H(y_0(0),c_0^*,T_0^*) - J_H(y_{K+1}(t_{K+1}),c_{K+1}^*,T_{K+1}^*) \\
            &\leq J_H(y_0(0),c_0^*,T_0^*) 
            \leq \changed{ \hat{V}_H(0,y_0(0)).}
        \end{align*}
        \end{subequations}
        Since the costs are nonnegative, the sum monotonically increases and is bounded for~$K\to\infty$. 
        So, it converges.
	\end{proof}
    We prove boundedness of the optimized parameters. 
	\begin{corollary} \label[corollary]{cor:parametersBounded}
        Let the assumptions of \Cref{lemma:FMPCAlgRecFeasible} be fulfilled. 
        Then, the set 
		$\left\{c_i^* \mid i\in\N_0 \right\} \subset \R$
		resulting from \Cref{alg:fmpc} is bounded. 
	\end{corollary}
	\begin{proof}
        According to \eqref{eq:valuefcn} and \eqref{eq:boundedCosts:descent} for all~ $i\in\N$ it holds
    	\begin{equation*}
    	    c_i^* \leq  J_H(y_i(t_i),c_i^*,T_i^*) \leq J_H(y_{i-1}(t_{i-1}),c_{i-1}^*,T_{i-1}^*).
    	\end{equation*}
        Thus, $c_i^*\leq J_H(y_0(0),c_0^*,T_0^*)$ for all $i\in\N$.
	\end{proof}
    
    Next, we show that the boundedness of the closed-loop costs implies convergence of the closed-loop solution, if~$Q$ is positive definite.
    We thereby prove that \Cref{alg:fmpc} fulfills the control objective~\eqref{eq:ControlObjective} in \Cref{sec:systemClassAndControlObjective}.
	In the context of \Cref{alg:fmpc}, define the \emph{closed-loop solution} as 
	\begin{subequations}			
		\begin{align*}
			y^{cl}(t) & \defeq \sum_{i=1}^\infty \chi_{[t_i,t_{i+1})}(t)\, y_i(t), \\
			u^{cl}(t) & \defeq \sum_{i=1}^\infty \chi_{[t_i,t_{i+1})}(t)\, u_i(t),
		\end{align*}
	\end{subequations}
	where $\chi_{[t_i,t_{i+1})}$ is the indicator function on~$[t_i,t_{i+1})$. 
    
	\begin{theorem} \label[theorem]{thm:TerminalConstraintsConvergence}
        Let the assumptions of \Cref{lemma:FMPCAlgRecFeasible} be fulfilled. 
        In addition, assume $Q$ to be positive definite. 
        Then, \Cref{alg:fmpc} asymptotically stabilizes system \labelcref{eq:system}, i.e., it holds that
        $\lim_{t\to \infty} y^{cl}(t) =0$.
	\end{theorem}
	\begin{proof} 
        Seeking a contradiction, we assume 
		\begin{equation}\label{eq:TerminalConstraintsStabilityProof:nonconvergence}
			\exists\, \varepsilon >0 \ 
			\forall\, t_0 \geq 0 \ 
			\exists\, t > t_0 \colon \ 
			\|y^{cl}(t)\| > \varepsilon. 
		\end{equation}

        \emph{Step 1}: 
        Let $ h = t_{i+1} - t_i $.
        Since $Q$ is symmetric positive definite, there exists a constant $q>0$ such that 
        $
            \langle y, Q y \rangle \geq q \|y\|^2
        $
        for all $y \in \R^m$.
		For $ t \geq 0 $, define the closed-loop funnel function as  
        $\varphi^{cl}(t) \defeq \sum_{i=1}^{\infty} \chi_{[t_i,t_{i+1})}(t) \;\varphi(t-t_i; c_i^*, T_i^*)$. 
		For $x\in\R_{\geq 0}$, define 
        $\hat{N}(x) \defeq \max_{\xi\in\interval{0}{x}} \left| N(\xi) \right| $.

        \emph{Step 2}: We show that, given the assumption \eqref{eq:TerminalConstraintsStabilityProof:nonconvergence}, the control input $u^{cl}$  is unbounded. 
		According to \Cref{thm:boundedCosts}, $J^{cl}_\infty(H,y_0(0)) < \infty$. This implies 
		\begin{equation} \label{eq:TerminalConstraintsStabilityProof:boundedCostsImplication}
			\forall\, \varepsilon_c > 0 \
			\exists\, t_{\varepsilon_c} \in \R\colon 
			\int_{t_{\varepsilon_c}}^{\infty} \ell(y^{cl}(t),u^{cl}(t)) \;\text{d}t < \varepsilon_c
		\end{equation}
		and, since $Q$ is positive definite, 
		\begin{equation}\label{eq:TerminalConstraintsStabilityProof:boundedCostsDecay} 
			\forall\, t\ge 0\ \exists\, \tau_{\nicefrac{\varepsilon}{3}}(t)>t\colon \ 
			\|y^{cl}(\tau_{\nicefrac{\varepsilon}{3}}(t))\| < \tfrac{\varepsilon}{3}.
		\end{equation}
		So, combining \eqref{eq:TerminalConstraintsStabilityProof:nonconvergence}~--~\eqref{eq:TerminalConstraintsStabilityProof:boundedCostsDecay}, 
        
		for any given \mbox{$\varepsilon_c\in \interval[open]{0}{\frac{qh\varepsilon^2}{36}}$}, 
		we find $t_\varepsilon > \tau_{\nicefrac{\varepsilon}{3}}(t_{\varepsilon_c})$ 
		with $\|y^{cl}(t_\varepsilon)\| > \varepsilon$. 
            There exist
		\begin{align*}
			t_{\nicefrac{1}{3}} 
			&\defeq \max\{ t\in\R_{\geq 0} \mid \|y^{cl}(t)\| = \tfrac{1}{3}\varepsilon \;\wedge\; t<t_\varepsilon \}, \\
			t_{\nicefrac{2}{3}} 
			&\defeq \min\{ t\in\R_{\geq 0} \mid \|y^{cl}(t)\| = \tfrac{2}{3}\varepsilon \;\wedge\; t>t_{\nicefrac{1}{3}} \} 
		\end{align*}
		because $y^{cl}$ is continuous and $t_\varepsilon > \tau_{\nicefrac{\varepsilon}{3}}(t_{\varepsilon_c})$. 
		  We estimate 
		\begin{align*}
			\varepsilon_c 
			&\geq \int_{t_{\nicefrac{1}{3}}}^{t_{\nicefrac{2}{3}}} \ell(y^{cl}(t),u^{cl}(t)) \;\text{d}t
			\geq \int_{t_{\nicefrac{1}{3}}}^{t_{\nicefrac{2}{3}}} \langle y^{cl}(t) , Q y^{cl}(t) \rangle \;\text{d}t\\
			&\geq \int_{t_{\nicefrac{1}{3}}}^{t_{\nicefrac{2}{3}}} q\| y^{cl}(t) \|^2 \;\text{d}t
			\geq q \left( \frac{\varepsilon}{3} \right)^2 (t_{\nicefrac{2}{3}} - t_{\nicefrac{1}{3}}).
		\end{align*}
            Since  $y^{cl}$ is piecewise differentiable, the mean value theorem 
            yields the existence of $t_\text{crit} \in \interval{t_{\nicefrac{1}{3}}}{t_{\nicefrac{2}{3}}}$ with
		\begin{align*}
		      \dd{t}\|{y}^{cl}(t_\text{crit})\| \geq 
			\tfrac{\frac{2}{3}\varepsilon - \frac{1}{3}\varepsilon}{t_{\nicefrac{2}{3}} - t_{\nicefrac{1}{3}}} 
			= \tfrac{\varepsilon}{3(t_{\nicefrac{2}{3}} - t_{\nicefrac{1}{3}})} 
                \geq \tfrac{3\varepsilon^3q}{\varepsilon_c}
			\xrightarrow[\varepsilon_c \to 0]{} \infty.
		\end{align*}
		Continuity of~$f$ and $ \| y^{cl}(t_\text{crit}) \| \leq \frac{2}{3}\varepsilon < \infty$ imply
		\begin{align} \label{eq:TerminalConstraintsStabilityProof:explodingInput}
			\| u^{cl}(t_\text{crit}) \| \xrightarrow[\varepsilon_c \to 0]{} \infty.
		\end{align}
        
        \emph{Step 3}:
        We will lead statement \eqref{eq:TerminalConstraintsStabilityProof:explodingInput}
        to a contradiction by establishing a bound on $ \| u^{cl}(t_\text{crit}) \|$.
        To this end, we distinguish two cases. 
        
        \textbf{Case 1:} $\varphi^{cl}(t_\text{crit}) > \varepsilon$. Since $\|y^{cl}(t_\text{crit})\| \leq \frac{2}{3}\varepsilon$ and $\varphi^{cl}(t_\text{crit}) > \varepsilon$, 
        we have $\frac{\|y^{cl}(t_\text{crit})\|}{\varphi^{cl}(t_\text{crit})} < \frac{2}{3}$.
		Then, \changed{with $c$ being the active parameter at time $t_\text{crit}$,}
		\begin{align*}
			\| u^{cl}(t_\text{crit}) \| 
			&= \left\| N \left( \changed{\alpha_c} \left(  \tfrac{\|y^{cl}(t_\text{crit})\|^2}{\varphi^{cl}(t_\text{crit})^2} \right) \right)  \tfrac{y^{cl}(t_\text{crit})}{\varphi^{cl}(t_\text{crit})} \right\| \\
			&\leq \hat{N} \left( \changed{\alpha_c} \left( \tfrac{4}{9}  \right) \right) \tfrac{2}{3}
			<\infty
		\end{align*}
		contradicting \eqref{eq:TerminalConstraintsStabilityProof:explodingInput}.

		\textbf{Case 2:} $\varphi^{cl}(t_\text{crit}) \leq \varepsilon$. 
            In this case, a similar chain of arguments can be performed. 
		Define 
			$
            \hat{t}_{\nicefrac{1}{3}} 
			\defeq \min\{ t\in\R_{\geq 0} \mid \|y^{cl}(t)\| = \tfrac{1}{3}\varepsilon \;\wedge\; t>t_\varepsilon \}, $ 
			and $
			\hat{t}_{\nicefrac{2}{3}} 
			\defeq \max\{ t\in\R_{\geq 0} \mid \|y^{cl}(t)\| = \tfrac{2}{3}\varepsilon \;\wedge\; t<\hat{t}_{\nicefrac{1}{3}} \}$, where
        $t_{\nicefrac{1}{3}}$ exists due to~\labelcref{eq:TerminalConstraintsStabilityProof:boundedCostsDecay}, 
			and~$t_{\nicefrac{2}{3}}$ due to the intermediate value theorem.

        We estimate
        \begin{align*}
            \varepsilon_c 
            & \geq \int_{\hat{t}_{\nicefrac{2}{3}}}^{\hat{t}_{\nicefrac{1}{3}}} \ell(y^{cl}(t),u^{cl}(t)) \;\text{d}t\\
            & \geq \int_{\hat{t}_{\nicefrac{2}{3}}}^{\hat{t}_{\nicefrac{1}{3}}} y^{cl}(t)^\T Q y^{cl}(t) \;\text{d}t
            \geq q \left( \frac{\varepsilon}{3} \right)^2 (\hat{t}_{\nicefrac{1}{3}} - \hat{t}_{\nicefrac{2}{3}}).
        \end{align*}
        The mean value theorem yields the existence of $\hat{t}_\text{crit} \in \interval{\hat{t}_{\nicefrac{2}{3}}}{\hat{t}_{\nicefrac{1}{3}}}$ with
		\begin{align*}
		      \dd{t}\|{y}^{cl}(t_\text{crit})\| \leq 
			\tfrac{\frac{1}{3}\varepsilon - \frac{2}{3}\varepsilon}{\hat{t}_{\nicefrac{1}{3}} - \hat{t}_{\nicefrac{2}{3}}} 
			= \tfrac{-\varepsilon}{3(\hat{t}_{\nicefrac{1}{3}} - \hat{t}_{\nicefrac{2}{3}})} 	
		      \leq \tfrac{-\varepsilon^3q}{\varepsilon_c} 	
			\xrightarrow[\varepsilon_c \to 0]{} -\infty.
		\end{align*}
		Since $f$ is continuous and $ \| y^{cl}(\hat{t}_\text{crit}) \| \leq \frac{2}{3}\varepsilon < \infty$, this implies 
        \changed{$\| u^{cl}(\hat{t}_\text{crit}) \| \to \infty$ for~$\varepsilon_c \to 0$.}

            \noindent
            \emph{\changed{Step 4}}:
		We show $\varphi^{cl}(\hat{t}_\text{crit}) > \frac{3}{4}\varepsilon$. 
		Seeking a contradiction, assume $\varphi^{cl}(\hat{t}_\text{crit}) \leq \frac{3}{4}\varepsilon$. 
		Since $\varphi^{cl}(t_\text{crit}) \leq \varepsilon$ and $\varphi^{cl}(t_\varepsilon) > \varepsilon$, there must be a resampling time between $t_\text{crit}$ and $t_\varepsilon$, i.e., 
        there exists~$i^* \in \N$ such that~$t_{i^*} \in \interval[open right]{t_\text{crit}}{t_\varepsilon}$.

            W.l.o.g. $T_i^* > h$ (otherwise we have convergence, since $ y^{cl}(t) = 0 $ for $t\geq t_{i^*}+h$).
		Since $\varphi^{cl}(t_{i^*}) > \varepsilon$,
		
            $\varphi^{cl}(t_{i^*}+\tau) > \varepsilon \left( 1 - \tfrac{\tau}{h} \right) $
            for  $\tau \in \interval[open right]{0}{h}$.
		
		This implies, if $\varphi^{cl}(\hat{t}_\text{crit}) \leq \frac{3}{4}\varepsilon$, then 
		$ \hat{t}_\text{crit} - t_{i^*} > \frac{h}{4} $. 
		But $ \| y^{cl}(t) \| \geq \frac{\varepsilon}{3} $ for $t\in\interval{t_{i^*}}{\hat{t}_\text{crit}} \subseteq \interval{t_{\nicefrac{1}{3}}}{\hat{t}_{\nicefrac{1}{3}}}$.
        This leads to 
		\begin{align*}
			\varepsilon_c
			& > \int_{t_{i^*}}^{\infty} \ell (y^{cl}(t), u^{cl}(t)) \;\text{d}t
			     > q\int_{t_{i^*}}^{\hat{t}_\text{crit}} \| y^{cl}(t) \|^2 \;\text{d}t \\
			& > q \, (\hat{t}_\text{crit} - t_{i^*}) \left( \tfrac{\varepsilon}{3} \right)^2 
    			 > q \; \tfrac{h}{4} \left( \tfrac{\varepsilon}{3} \right)^2 
    			 > \varepsilon_c,
		\end{align*}
        a contradiction.  Therefore, $ \| y^{cl} (\hat{t}_\text{crit}) \| \leq \frac{2}{3}\varepsilon$
        and $\varphi^{cl}(\hat{t}_\text{crit}) > \frac{3}{4}\varepsilon$.
        It follows that, \changed{with $c$ being the active parameter at time $\hat{t}_\text{crit}$,}
		\begin{align*}
			\| u^{cl}(\hat{t}_\text{crit}) \| 
			&= \left\| N \left( \changed{\alpha_c} \left(  \tfrac{\|y^{cl}(\hat{t}_\text{crit})\|^2}{\varphi^{cl}(\hat{t}_\text{crit})^2} \right) \right)  \tfrac{y^{cl}(\hat{t}_\text{crit})}{\varphi^{cl}(\hat{t}_\text{crit})} \right\| \\
			&\leq \hat{N} \left( \changed{\alpha_c} \left( \tfrac{8^2}{9^2} \right) \right) \tfrac{8}{9} <\infty,
		\end{align*}
        \changed{which contradicts $\| u^{cl}(\hat{t}_\text{crit}) \| \to \infty$ for~$\varepsilon_c \to 0$.}

	\end{proof}
    \section{Numerical Example}
    
    We illustrate \Cref{alg:fmpc} 
    
    considering the 
    system
    \begin{equation}\label{eq:ExampleSys}
        \begin{small}
        \dot{y}(t) = \begin{pmatrix}
            y_1(t)^2 + y_1(t) \\ y_2(t)^2 + y_1(t)
        \end{pmatrix}  - u(t) , \quad
        y(0) = \begin{pmatrix*}[r]3\\-3\end{pmatrix*}.
        \end{small}
    \end{equation}
    In the stage costs~\eqref{eq:stageCosts}, we set $Q=I_2$ and $R=0.2\cdot I_2$, where $I_2 \in \R^{2 \times 2}$ denotes the identity matrix.
    
    As control parameters, we choose $ N(s) = s $ (which is possible due to known control direction \cite[Rem.~2.5]{lanza2024exact}) and $\alpha_c(s) = \tfrac{2c}{1-s}$, and set~$\psi(t) = \infty $ \changed{(disabled outer funnel).}
    
    As described in \cite[Sec.~4.1]{lanza2024exact}, because $\lim_{t \to T} \varphi(t) = 0 $, simulation is only possible on an interval \changed{$[t_i, t_i + t_\text{max}]$} with $t_\text{max}<T $. 
    
    We choose $t_\text{max} = T-10^{-9}/c $ to guarantee a spatial accuracy of $10^{-9}$ \changed{in the prediction}. 
    
    \changed{The closed-loop system is simulated 
    on the interval~$[0,3]$.
    }
    
    The prediction horizon is chosen as $H=5$ and the step size as $h=0.25$. 
    This numerical experiment was performed in \textsc{Matlab R2024b}, using \texttt{ode45} ($\texttt{AbsTol}=10^{-9}$, $\texttt{RelTol}=10^{-6}$) and \texttt{fmincon} as optimization algorithm (using default parameterization).
    The closed-loop system output norm $\|y^{cl}\|_2$, as well as the closed-loop funnel $\varphi^{cl}$ is plotted in \Cref{fig:sim:output}. 
    \changed{
    We observe that the system asymptotically converges to the equilibrium. 
    Furthermore, although no outer funnel is active, the product $c_i^*  T_i^*$ decreases in each step, meaning 
    $\varphi(0;c^*_{i+1},T_{i+1}^*) < \varphi(0;c^*_{i},T_{i}^*)$. 
    }
    \changed{
    The input is shown in \Cref{fig:sim:input}, where the previous observation is visible in terms of increasing input values whenever the state is close to the funnel boundary.
    }
    \begin{figure}[hbtp]
    \begin{subfigure}[b]{0.48\textwidth}
        \centering
        \includegraphics[width=.9\linewidth]{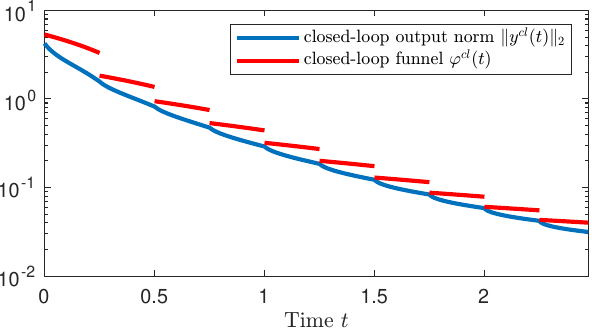}
        \caption{Output norm and funnel.}
        \label{fig:sim:output}
    \end{subfigure}
      \begin{subfigure}[b]{0.48\textwidth}
        \centering
        \includegraphics[width=.9\linewidth]{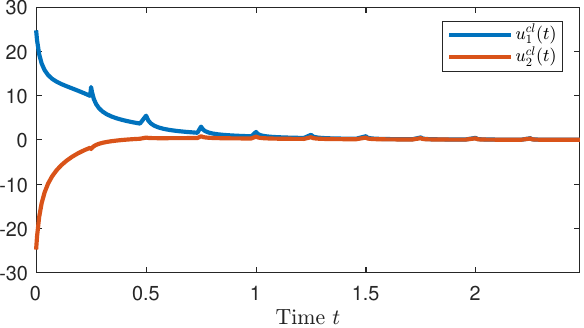}
        \caption{Control input.}
        \label{fig:sim:input}
    \end{subfigure}
    \caption{Simulation of system~\eqref{eq:ExampleSys} under the control generated by~\Cref{alg:fmpc}.}
    \end{figure}
    
    \section{Conclusion and Outlook}
    We presented model predictive funnel control, a novel combination of funnel control and MPC. 
    It features a constant number of decision variables in the optimization problem and yields 
    \changed{closed-loop guarantees}
    
    even on inter-sampling intervals. 
    We rigorously showed boundedness of the closed-loop costs, as well as convergence of the solution. 
    Future work will address systems of higher relative degree with internal dynamics, relaxing the choice of funnel boundary ($\varphi(T) \neq 0$), considering input constraints, and performing reference tracking tasks, to name a few of the upcoming topics. 
    \changed{We will conduct numerical studies to analyse performance in the sense of computation time compared to MPC in specific scenarios, e.g., for long prediction horizons or fine discretization grids. }

\bibliographystyle{IEEEtran}
\bibliography{references}

% Generated by IEEEtran.bst, version: 1.14 (2015/08/26)
\begin{thebibliography}{10}
\providecommand{\url}[1]{#1}
\csname url@samestyle\endcsname
\providecommand{\newblock}{\relax}
\providecommand{\bibinfo}[2]{#2}
\providecommand{\BIBentrySTDinterwordspacing}{\spaceskip=0pt\relax}
\providecommand{\BIBentryALTinterwordstretchfactor}{4}
\providecommand{\BIBentryALTinterwordspacing}{\spaceskip=\fontdimen2\font plus
\BIBentryALTinterwordstretchfactor\fontdimen3\font minus
  \fontdimen4\font\relax}
\providecommand{\BIBforeignlanguage}[2]{{%
\expandafter\ifx\csname l@#1\endcsname\relax
\typeout{** WARNING: IEEEtran.bst: No hyphenation pattern has been}%
\typeout{** loaded for the language `#1'. Using the pattern for}%
\typeout{** the default language instead.}%
\else
\language=\csname l@#1\endcsname
\fi
#2}}
\providecommand{\BIBdecl}{\relax}
\BIBdecl

\bibitem{Ilchmann2002}
A.~Ilchmann, E.~P. Ryan, and C.~J. Sangwin, ``Tracking with prescribed
  transient behaviour,'' \emph{ESAIM: Control Optim. Calc. Var.}, vol.~7, pp.
  471--493, 2002.

\bibitem{berger2021funnel}
T.~Berger, A.~Ilchmann, and E.~P. Ryan, ``Funnel control of nonlinear
  systems,'' \emph{Math. Control Signals Syst.}, vol.~33, pp. 151--194, 2021.

\bibitem{Hackl2017}
C.~M. Hackl, \emph{Non-identifier Based Adaptive Control in
  Mechatronics}.\hskip 1em plus 0.5em minus 0.4em\relax Springer International
  Publishing, 2017.

\bibitem{bechlioulis2008robust}
C.~P. Bechlioulis and G.~A. Rovithakis, ``Robust adaptive control of feedback
  linearizable mimo nonlinear systems with prescribed performance,'' \emph{IEEE
  Trans. Autom. Control}, vol.~53, no.~9, pp. 2090--2099, 2008.

\bibitem{bechlioulis2014low}
------, ``A low-complexity global approximation-free control scheme with
  prescribed performance for unknown pure feedback systems,''
  \emph{Automatica}, vol.~50, no.~4, pp. 1217--1226, 2014.

\bibitem{namerikawa2024equivalence}
R.~Namerikawa, A.~Wiltz, F.~Mehdifar, T.~Namerikawa, and D.~V. Dimarogonas,
  ``On the equivalence between prescribed performance control and control
  barrier functions,'' in \emph{Proc. American Control Conference (ACC)}, 2024,
  pp. 2458--2463.

\bibitem{lee2019asymptotic}
J.~G. Lee and S.~Trenn, ``Asymptotic tracking via funnel control,'' in
  \emph{2019 IEEE 58th Conference on Decision and Control (CDC)}, 2019, pp.
  4228--4233.

\bibitem{lanza2024exact}
L.~Lanza, ``Exact output tracking in prescribed finite time via funnel
  control,'' \emph{Automatica}, vol. 170:~111873, 2024.

\bibitem{Lee2011}
J.~H. Lee, ``Model predictive control: Review of the three decades of
  development,'' \emph{Int. J. Control Autom. Syst.}, vol.~9, no.~3, 2011.

\bibitem{badgwell2021model}
T.~A. Badgwell and S.~J. Qin, ``Model predictive control in practice,'' in
  \emph{Encyclopedia of Systems and Control}.\hskip 1em plus 0.5em minus
  0.4em\relax Springer, 2021, pp. 1239--1252.

\bibitem{jerez2014embedded}
J.~L. Jerez, P.~J. Goulart, S.~Richter, G.~A. Constantinides, E.~C. Kerrigan,
  and M.~Morari, ``Embedded online optimization for model predictive control at
  megahertz rates,'' \emph{IEEE Trans. Autom. Control}, vol.~59, no.~12, pp.
  3238--3251, 2014.

\bibitem{gros2020linear}
S.~Gros, M.~Zanon, R.~Quirynen, A.~Bemporad, and M.~Diehl, ``From linear to
  nonlinear {MPC}: bridging the gap via the real-time iteration,'' \emph{Int.
  J. Control}, vol.~93, no.~1, pp. 62--80, 2020.

\bibitem{mayne2005robust}
D.~Q. Mayne, M.~M. Seron, and S.~Rakovi{\'c}, ``Robust model predictive control
  of constrained linear systems with bounded disturbances,'' \emph{Automatica},
  vol.~41, no.~2, pp. 219--224, 2005.

\bibitem{mayne2011tube}
D.~Q. Mayne, E.~C. Kerrigan, E.~Van~Wyk, and P.~Falugi, ``Tube-based robust
  nonlinear model predictive control,'' \emph{Int. J. Robust \& Nonlinear
  Control}, vol.~21, no.~11, pp. 1341--1353, 2011.

\bibitem{limon2006input}
D.~Lim{\'o}n, T.~Alamo, F.~Salas, and E.~F. Camacho, ``Input to state stability
  of min--max mpc controllers for nonlinear systems with bounded
  uncertainties,'' \emph{Automatica}, vol.~42, no.~5, pp. 797--803, 2006.

\bibitem{xie2024data}
Y.~Xie, J.~Berberich, and F.~Allg{\"o}wer, ``Data-driven min-max mpc for linear
  systems: Robustness and adaptation,'' \emph{arXiv preprint arXiv:2404.19096},
  2024.

\bibitem{keerthi1988optimal}
S.~S. Keerthi and E.~G. Gilbert, ``Optimal infinite-horizon feedback laws for a
  general class of constrained discrete-time systems: Stability and
  moving-horizon approximations,'' \emph{Journal of optimization theory and
  applications}, vol.~57, pp. 265--293, 1988.

\bibitem{berger2020learning}
T.~Berger, C.~K{\"a}stner, and K.~Worthmann, ``Learning-based funnel-mpc for
  output-constrained nonlinear systems,'' \emph{IFAC-PapersOnLine}, vol.~53,
  no.~2, pp. 5177--5182, 2020.

\bibitem{BergDenn21}
T.~Berger, D.~Dennst\"{a}dt, A.~Ilchmann, and K.~Worthmann, ``Funnel {M}odel
  {P}redictive {C}ontrol for {N}onlinear {S}ystems with {R}elative {D}egree
  {O}ne,'' \emph{SIAM Journal on Control and Optimization}, vol.~60, no.~6, pp.
  3358--3383, 2022.

\bibitem{Berger2024FMPC}
T.~Berger and D.~Dennstädt, ``Funnel {MPC} for nonlinear systems with
  arbitrary relative degree,'' \emph{Automatica}, vol. 167, 2024.

\bibitem{kohler2022constrained}
J.~Köhler, M.~A. Müller, and F.~Allgöwer, ``Constrained nonlinear output
  regulation using model predictive control,'' \emph{IEEE Trans. Autom.
  Control}, vol.~67, no.~5, pp. 2419--2434, 2022.

\end{thebibliography}
	
\end{document}